\documentclass[11pt,a4paper]{article}

\usepackage{amssymb,amsmath}
\usepackage{enumerate}
\usepackage{latexsym}
\usepackage{amscd}
\usepackage{amsthm}
\usepackage{mathrsfs}
\usepackage[margin=3cm]{geometry}

\newtheorem{theorem}{Theorem}[section]
\newtheorem{corollary}[theorem]{Corollary}
\newtheorem{lemma}[theorem]{Lemma}
\newtheorem{definition-lemma}[theorem]{Definition-Lemma}

\newtheorem{proposition}[theorem]{Proposition}
\newtheorem{conjecture}[theorem]{Conjecture}

\theoremstyle{definition}
\newtheorem{definition}[theorem]{Definition}
\newtheorem{remark}[theorem]{Remark}

\numberwithin{equation}{section}  

\usepackage[all]{xy}
\usepackage{latexsym}

\newcommand{\relmid}{\mathrel{}\middle|\mathrel{}}

\newdimen\argwidth
\def\db[#1\db]{%
	\setbox0=\hbox{$#1$}\argwidth=\wd0
	\setbox0=\hbox{$\left[\box0\right]$}
	\advance\argwidth by -\wd0
	\left[\kern.3\argwidth\box0 \kern.3\argwidth\right]}

\newcommand{\id}{\operatorname{id}}

\newcommand{\Supp}{\operatorname{Supp}}

\newcommand{\Hom}{\operatorname{Hom}}

\newcommand{\Spec}{\operatorname{Spec}}

\newcommand{\cHom}{\mathop{{\cH}om}\nolimits}

\newcommand{\Cone}{\operatorname{Cone}}

\newcommand{\Module}{\operatorname{Mod}}

\newcommand{\Fuk}{\operatorname{\mathfrak{Fuk}}}

\newcommand{\bA}{\ensuremath{\mathbb{A}}}

\newcommand{\bC}{\ensuremath{\mathbb{C}}}
\newcommand{\bD}{\ensuremath{\mathbb{D}}}

\newcommand{\bR}{\ensuremath{\mathbb{R}}}

\newcommand{\bZ}{\ensuremath{\mathbb{Z}}}

\newcommand{\cA}{\ensuremath{\mathcal{A}}}
\newcommand{\cB}{\ensuremath{\mathcal{B}}}
\newcommand{\cC}{\ensuremath{\mathcal{C}}}

\newcommand{\cE}{\ensuremath{\mathcal{E}}}
\newcommand{\cF}{\ensuremath{\mathcal{F}}}
\newcommand{\cG}{\ensuremath{\mathcal{G}}}
\newcommand{\cH}{\ensuremath{\mathcal{H}}}

\newcommand{\cL}{\ensuremath{\mathcal{L}}}

\newcommand{\cN}{\ensuremath{\mathcal{N}}}
\newcommand{\cO}{\ensuremath{\mathcal{O}}}

\newcommand{\cS}{\ensuremath{\mathcal{S}}}

\newcommand{\frakC}{\ensuremath{\mathfrak{C}}}

\newcommand{\lto}{\longrightarrow}

\newcommand{\simto}{\xrightarrow{\sim}}

\newcommand{\la}{\left\langle}
\newcommand{\ra}{\right\rangle}

\newcommand{\lc}{\left\{}
\newcommand{\rc}{\right\}}

\newcommand{\Sh}{\operatorname{\mathbf{Sh}}}
\newcommand{\cSh}{\operatorname{\Sh}^c}
\newcommand{\lSh}{\operatorname{\Sh}^\diamondsuit}
\newcommand{\wSh}{\operatorname{\Sh}^w}

\DeclareMathOperator{\Coh}{{\operatorname{\mathbf{coh}}}}
\DeclareMathOperator{\QCoh}{{\operatorname{\mathbf{Qcoh}}}}

\newcommand{\tl}[1]{\widetilde{#1}}

\DeclareMathOperator{\Int}{\mathrm{Int}}
\DeclareMathOperator{\Perf}{\mathrm{\mathbf{Perf}}}
\newcommand{\Ls}{{\Lambda_\Sigma}}
\newcommand{\Us}{{U_\sigma}}
\renewcommand{\cC}{\wSh_\Ls(T^n)}
\renewcommand{\cN}{{\mathcal{\mathbf{N}}_\rho}}
\newcommand{\ks}{{\kappa_\Sigma}}
\renewcommand{\Hom}{{\mathop{\mathrm{hom}}}}

\DeclareMathOperator*{\colim}{\operatorname{colim}}
\newcommand{\mC}{{\mathcal{C}}}
\newcommand{\tLs}{{\widetilde{\Ls}}}
\newcommand{\bs}{{\backslash}}

\newcommand{\Ind}{{\operatorname{Ind}}}
\DeclareMathOperator*{\indlim}{``\varinjlim"}
\newcommand{\MS}{{\operatorname{SS}}}

\DeclareMathOperator{\plushat}{\hat{+}}

\let\oldappendix\appendix
\renewcommand{\appendix}{%
	\oldappendix%
	\renewcommand{\thesection}{\Alph{section}}}%

\title{Categorical Localization for the \\ Coherent-Constructible Correspondence}
\author{Yuichi Ike\footnote{Yuichi Ike; 
		Fujitsu Laboratories Ltd., 4-1-1 Kamikodanaka, Nakahara-ku, Kawasaki, Kanagawa, 211-8588 Japan; \newline
		e-mail: \texttt{yuichi.ike.1990@gmail.com}, \texttt{ike.yuichi@jp.fujitsu.com}} 
	\and 
	Tatsuki Kuwagaki\footnote{Tatsuki Kuwagaki; 
		Kavli Institute for Physics and Mathematics of the Universe (WPI), 5-1-5 Kashiwanoha, Kashiwa-shi,
		Chiba, 277-8583, Japan; \newline
		e-mail: \texttt{tatsuki.kuwagaki@ipmu.jp}}}
\date{\today}

\begin{document}
\maketitle

\begin{abstract}
	We prove a microlocal counterpart of categorical localization for Fukaya categories in the setting of the coherent-constructible correspondence. 
	\medskip
	
	
\end{abstract}

\section{Introduction}
Kontsevich's homological mirror symmetry (HMS) conjecture (\cite{KonHMS}) states that two categories associated to a mirror pair are equivalent. For a Calabi--Yau (CY) variety, a mirror is also CY and the conjecture is a quasi-equivalence between the dg category of coherent sheaves over one and the (derived) Fukaya category of the other one.

For non-CYs, mirrors do not need to be varieties. For a Fano toric variety which we will focus on, its mirror is a Landau--Ginzburg (LG) model, which is a holomorphic function on $(\bC^\times)^n$ which can be read from the defining fan of the toric variety (\cite{HV}). Fukaya-type (A-brane) category for an LG model is known to be the Fukaya--Seidel category (\cite{Seidelmutation}) when the LG-model is a Lefschetz fibration. Hence, for a smooth Fano fan, HMS predicts a quasi-equivalence between the dg category of coherent sheaves $\Coh X_\Sigma$ over the toric variety $X_\Sigma$ and the Fukaya--Seidel category $\Fuk(W_\Sigma)$ of the Laurent function $W_\Sigma$:
\begin{equation}\label{HMS}
\Coh X_\Sigma\simeq \Fuk(W_\Sigma).
\end{equation}
This is now proved in many cases (e.g., \cite{AKO, AKO2, Ueda, UY}).

When a variety is not complete, its derived category of coherent sheaves has an infinite-dimensional nature (it has infinite-dimensional Hom spaces). Accordingly, its mirror Fukaya-type category also should have an infinite-dimensional nature. Such a construction is known to be (partially) wrapped Fukaya categories, whose hom spaces could have infinitely many generators (intersection points) formed by quadratically increasing Hamiltonian isotopy (\cite{ASviterbo, Sylvan}).

There is a relation between Fukaya--Seidel and wrapped Fukaya, i.e., the latter is obtained by categorical localization of the former. This point of view is due to Seidel (\cite{Seidelnat}). Let $X$ be a variety and $D$ be a divisor of $X$. Then the dg category of coherent sheaves over $X\bs D$ is obtained by the quotient of the dg category of $X$ by the objects supported on $D$:
\begin{equation}\label{loccoh}
\Coh (X\bs D)\simeq \Coh X/ \Coh_D X.
\end{equation} 
In \cite{Seidelnat}, he expected that the mirror operation of this is the identification of $\id$ and the Serre functor when $D$ is a canonical divisor. Hence, the wrapped Fukaya mirror to $X\bs D$ is obtained by this identification for the Fukaya--Seidel mirror to $X$ (\cite{ASlefshetz, AG}). Our point of view is more closely related to Sylvan's \cite{Sylvan}. He associated the partially wrapped Fukaya category $W_\mathbf{s}(M)$ for a symplectic manifold $M$ with symplectic stops $\mathbf{s}$ which designate directions not to be wrapped. Removing some of symplectic stops $\mathbf{r}$ allows Lagrangians to wrap to the corresponding directions. The resulting category can be described as 
\begin{equation}\label{locfuk}
W_{\mathbf{s}\bs\mathbf{r}}(M)\simeq W_{\mathbf{s}}(M)/\cB_{\mathbf{r}}
\end{equation}
where $\cB_{\mathbf{r}}$ is the full subcategory spanned by Lagrangians near $\mathbf{r}$.

Our aim in this paper is relating (\ref{loccoh}) and an analogue of (\ref{locfuk}) in the microlocal world.
Microlocal study of A-brane categories (Fukaya-type categories) started from the work of Nadler and Zaslow (\cite{Nad, NZ}). Their theorem shows that the infinitesimally wrapped Fukaya category of a cotangent bundle is quasi-equivalent to the dg category of constructible sheaves over its base space. This theorem motivates us to use some types of the dg category of constructible sheaves for A-branes in homological mirror symmetry. For mirrors of toric varieties, Fang--Liu--Treumann--Zaslow (\cite{FLTZ, FLTZ2}) provide the following candidates for such categories as the full subcategories of the dg categories of the constructible sheaves over real tori.

Let $M$ be a free abelian group of rank $n$ and $N$ be its dual. 
Further, let $\Sigma$ be a smooth (not necessarily complete) fan defined in $N_\bR:=N\otimes_\bZ\bR$. 
Set $T^n:=M_\bR/M$. 
We identify $T^n \times N_\bR$ with the cotangent bundle $T^*T^n$ and set 
\begin{equation}
\Ls:=\bigcup_{\sigma\in \Sigma}p(\sigma^\perp)\times (-\sigma)\subset T^*T^n,
\end{equation}
where $p\colon M_\bR\rightarrow T^n$ is the quotient map and $\sigma^\perp\subset M_\bR$ is the orthogonal subspace of $\sigma\subset N_\bR$. It is known that there is a fully faithful morphism $\ks\colon \Coh X_\Sigma\rightarrow \lSh_\Ls(T^n)$ where the right-hand side is the full sub dg category of the quasi-constructible sheaves over $T^n$ spanned by objects whose microsupports are contained in $\Ls$ (see \cite{FLTZ, FLTZ2}).

\begin{conjecture}[The coherent-constructible correspondence \cite{FLTZ, FLTZ2}]\label{introccc}
	Suppose that $\Sigma$ is complete. Then $\ks$ induces a quasi-equivalence
	\begin{equation}
	\Coh(X_\Sigma)\simeq \cSh_\Ls(T^n)
	\end{equation}
	where the right-hand side is the full sub dg category of the constructible sheaves over $T^n$ spanned by objects whose microsupports are contained in $\Ls$.
\end{conjecture}
This is considered to be a microlocal counterpart of \eqref{HMS}. In the course of this paper, we will assume a slightly stronger version for (a not necessarily complete) $\Sigma$.

\begin{conjecture}\label{introqccc}
	The morphism $\ks$ extends to a quasi-equivalence
	\begin{equation}
	\QCoh X_\Sigma\simeq \lSh_\Ls(T^n)
	\end{equation}
	where the left-hand side is the dg category of quasi-coherent sheaves over $X_\Sigma$.
\end{conjecture}
Conjecture~\ref{introccc} is proved in some special cases (\cite{Tr, SS, Kuw}) and in the equivariant version (\cite{FLTZ}). Conjecture~\ref{introqccc} is sketched in \cite{Vaintrob} and also proved in \cite{Kuw2} independently. The former follows from the latter (Corollary~\ref{gccctoccc} below).

Conjecture~\ref{introqccc} induces an equivalence between the subcategories of compact objects. The compact objects of the left-hand side are coherent sheaves $\Coh X_\Sigma$ and those of the right-hand side are the wrapped constructible sheaves $\wSh_\Ls(T^n)$ of Nadler (\cite{nadler2016wrapped}), which are introduced as a microlocal counterpart of (partially) wrapped Fukaya categories.

To state an analogue of (\ref{locfuk}), we also need the notion of microlocal skyscraper sheaves which is also due to Nadler. Roughly speaking, for a point of a Lagrangian of a cotangent bundle, a microlocal skyscraper sheaf represents the microlocal stalk functor at the point. For a 1-dimensional cone $\rho\in \Sigma$, let $\Sigma_c^\rho$ be the complement of the set of cones containing $\rho$ in $\Sigma$. Set $\cB_\rho$ to be the full subcategory of $\wSh_\Ls(T^n)$ split-generated by microlocal skyscraper sheaves over regular points of $\Lambda_{\Sigma}\bs\Lambda_{\Sigma^\rho_c}$. 

Then our main theorem is the following.

\begin{theorem}\label{intromain}
	Assuming Conjecture~\ref{introqccc} for $\Sigma$, there is a quasi-equivalence
	\begin{equation}
	\wSh_{\Lambda_{\Sigma^\rho_c}}(M_\bR/M)\simeq \wSh_\Ls(M_\bR/M)/\cB_\rho.
	\end{equation}
	such that the induced morphism $\wSh_\Ls(M_\bR/M)\rightarrow \wSh_{\Lambda_{\Sigma^\rho_c}}(M_\bR/M)$ fits into the diagram
	\begin{equation}
	\xymatrix@C=40pt{
		\Coh X_\Sigma \ar[r]^-{\kappa_\Sigma}_-{\simeq} \ar[d] &\ar[d] \wSh_\Ls(M_\bR/M)\\
		\Coh X_{\Sigma^\rho_c} \ar[r]_-{\kappa_{\Sigma^\rho_c}}^-{\simeq} & \wSh_{\Lambda_{\Sigma_c^\rho}}(M_\bR/M).
	}
	\end{equation}
\end{theorem}

In Appendix~\ref{sec:surface}, we also prove the following.
\begin{theorem}
	If $\dim\Sigma=2$, Conjecture~\ref{introqccc} is true.
\end{theorem}
The proof is independent of that of \cite{Vaintrob, Kuw2} and is based on the fact Conjecture~\ref{introccc} is true for $\dim\Sigma=2$ (\cite{Kuw}) and an explicit proof that ``constructible = wrapped'' for 2-dimensional complete $\Sigma$.

The structure of the paper is as follows. In Section~\ref{sec:preliminaries}, we prepare microlocal sheaf theory and give a review of the coherent-constructible correspondence (CCC) for later sections. In Section~\ref{sec:main}, we will prove our main theorem. In Appendix~\ref{sec:surface}, we give a proof of Conjecture~\ref{introqccc} for $\dim \Sigma=2$.

\section*{Notation}
We denote the space of morphisms of dg categories by $\hom^\bullet(-,-)$.
In this paper, any (co)limits in dg categories are in a homotopical sense. Cocompleteness always means filtered cocompleteness, i.e., it has a colimit over any small filtered categories (cf.\ \cite{MR1827714}). All functors are derived. All morphisms between dg categories are meant to be morphisms in the homotopy category of the dg category of dg categories localized at quasi-equivalences. We freely use standard toric terminologies from \cite{CLS}, for example, smooth fan and complete fan.

\section{Preliminaries}\label{sec:preliminaries}

\subsection{Microlocal theory of sheaves}

Let $Z$ be a real analytic manifold.
We denote by $\lSh(Z)$ the dg category of complexes of sheaves of $\bC$-vector spaces on $Z$ whose cohomology sheaves are quasi-$\bR$-constructible in the sense of  \cite{FLTZ} (weakly constructible in the sense of \cite{KS}, large constructible in the sense of \cite{nadler2016wrapped}).
Let $\cSh(Z) \subset \lSh(Z)$ be the full dg subcategory consisting of $\bR$-constructible complexes.
In what follows, we consider only (quasi-) $\bR$-constructible sheaves and simply write (quasi-) constructible. 
\medskip

First we recall the definition of microsupports of sheaves (see Kashiwara--Schapira \cite[Section~5.1]{KS}). Although the treatment of \cite{KS} is limited to bounded complexes, we can obtain similar results for unbounded complexes by using \cite{Sp}. See, for example, \cite{Sch}.

\begin{definition}
	Take a point $z \in Z$.
	Let $f$ be a real analytic function defined in a neighborhood of $z$ satisfying $f(z)=0$.
	Recall that the local cohomology was defined as 
	\begin{equation}
	\Gamma_{\{f \ge 0\}}(F)_z
	\simeq
	\Gamma_{\{f \ge 0\}}(B;F)
	=
	\Cone(\Gamma(B;F) \lto \Gamma(B \cap \{f<0\}))[-1],
	\end{equation}
	where $B \subset Z$ is a sufficiently small open ball centered at $z$.

	For any $F \in \lSh(Z)$, we define its \emph{microsupport} (or \emph{singular support}) $\MS(F) \subset T^*Z$ as 
	\begin{align*}
	(z_0;\zeta_0) \not\in \MS(F) 
	\overset{\mathrm{def}}{\Longleftrightarrow} \quad &
	\text{there is an open neighborhood $U$ of $(z_0;\zeta_0)$ such that for} \\
	& \text{any $z \in Z$ and any $C^\infty$-function $f$ defined in a neighbor-} \\
	& \text{hood of $z$ with $(z;df(z)) \in U$, one has $\Gamma_{\{f \ge 0\}}(F)_z \simeq 0$.}
	\end{align*}
\end{definition}

Note that $\MS(F) \cap T^*_ZZ=\Supp(F)$.
For $F \in \lSh(Z)$, the microsupport $\MS(F)$ is a closed conic Lagrangian subvariety of $T^*Z$.
\medskip

The following proposition is a very important tool 
in the microlocal theory of sheaves (for the original version, see \cite[Proposition~2.7.2]{KS}).

\begin{proposition}[The non-characteristic deformation lemma]\label{nonchara}
	\
	Let $F \in \lSh(Z)$ and $\{U_t\}_{t \in \bR}$ be a family of subanalytic open subsets of $Z$.
	Assume the following conditions:
	\begin{enumerate}
		\item $U_t=\bigcup_{s<t}U_s \quad (t \in \bR)$.
		\item For any $(s,t) \in \bR^2$ with $s \le t$, 
		$\overline{U_t \setminus U_s} \cap \Supp(F)$ is compact.
		\item\footnote{The authors have a mistake in the definition of $Z_s$ in \cite[Proposition~2.7.2]{KS}. See the errata  in Schapira's page \texttt{https://webusers.imj-prg.fr/\~{}pierre.schapira/books/Errata.pdf}.}
		Setting $Z_s:=\bigcap_{t>s} \overline{(U_t \setminus U_s)}$, for any $(s,t) \in \bR^2$ with $s \le t$ and $z \in Z_s \setminus U_t$ we have
		\begin{equation}
		(\Gamma_{X \setminus U_t}(F))_z=0.
		\end{equation}
		Then for any $t \in \bR$, one has a quasi-isomorphism
		\begin{equation}
		\Gamma\Bigg(\bigcup_{t \in \bR}U_t;F\Bigg) \simto 
		\Gamma(U_t;F).
		\end{equation}
	\end{enumerate} 
\end{proposition}

By the definition of microsupports and the non-characteristic deformation lemma, we obtain the following proposition, called the microlocal Morse lemma.

\begin{proposition}[Microlocal Morse lemma]\label{MML}
	Let $F \in \lSh(Z)$ and $\psi \colon Z \to \bR$ be a real analytic function.
	Let $a<b$ in $\bR$ and assume that 
	\begin{enumerate}
		\item $\psi$ is proper on $\Supp(F)$;
		\item for any $z \in \psi^{-1}([a,b))$, $d\psi(z) \not\in \MS(F)$.
	\end{enumerate}
	Then the natural restriction morphism 
	\begin{equation}
	\Gamma(\psi^{-1}((-\infty,a));F)
	\lto 
	\Gamma(\psi^{-1}((-\infty,b));F)
	\end{equation}
	is a quasi-isomorphism.
\end{proposition}

We denote by $\cHom$ the internal Hom in the dg category $\lSh(Z)$.
We shall give a bound for $\cHom(F,G)$.

\begin{definition}[{\cite[Definition~6.2.3(v) and Remark~6.2.8(ii)]{KS}}]
	Let $A$ and $B$ be closed conic subsets of $T^*Z$.
	We define a conic subset $A \plushat B$ of $T^*Z$ by using a local coordinate system as follows.
	Let $x$ be a local coordinate system of $Z$ and $(x;\xi)$ be the associated coordinate system of $T^*Z$.
	Then $(x_0;\xi_0) \in A \plushat B$ if there exist sequences $\{(x_n;\xi_n)\}_n$ in $A$ and $\{(y_n;\eta_n)\}_n$ in $B$ such that $x_n \underset{n}{\to} x_0, y_n \underset{n}{\to} x_0, \xi_n + \eta_n \underset{n}{\to} \xi_0$, and $|x_n-y_n||\xi_n| \underset{n}{\to} 0$.
\end{definition}

\begin{proposition}[{\cite[Corollary~6.4.5(ii)]{KS}}]\label{prp:SSHom}
	For $F,G \in \lSh(Z)$, one has
	\begin{equation}
	\MS(\cHom(F,G)) \subset \MS(G) \plushat \MS(F)^a,
	\end{equation}
	where $a \colon T^*Z \to T^*Z, (x;\xi) \mapsto (x;-\xi)$ denotes the antipodal map.
\end{proposition}

We introduce the notion of a microlocal stalk (microlocal type in the sense of \cite{KS}).
Let $\Lambda \subset T^*Z$ be a conic Lagrangian subvariety and $p=(z_0; \zeta_0) \in \Lambda$ a point in $\Lambda$.
Assume that $\Lambda$ is smooth near $p$.
Take a real analytic function $f$ in a neighborhood of $z_0$ satisfying the following three conditions: 
\begin{enumerate}
	\item $f(z_0)=0$.
	\item $df(z_0)=p$.
	\item The graph $\Gamma_{df}$ and $\Lambda$ intersect transversely at $p$.
\end{enumerate}

\begin{definition}\label{microlocalstalk}
	Let $\Lambda, p, f$ be as above and $F \in \lSh(Z)$. Assume $\MS(F) \subset \Lambda$.
	We define the \emph{microlocal stalk} at $p$ associated with $f$ by 
	\begin{equation}\label{eq:microlocalstalk}
	\phi_{p,f}(F):=
	\Gamma_{\{f \ge 0\}}(F)_{z_0}
	\simeq 
	\Gamma_{\{f \ge 0\}}(B;F),
	\end{equation}
	where $B \subset Z$ is a sufficiently small open ball centered at $z_0$.
\end{definition}

Note that $B$ depends only on $\Lambda, p$, and $f$.
That is, we can take the same $B$ for any object $F$ with $\MS(F) \subset \Lambda$ in \eqref{eq:microlocalstalk}.

\subsection{Nadler's wrapped constructible sheaves}
For a closed conic Lagrangian subvariety $\Lambda \subset T^*Z$, we denote by $\lSh_{\Lambda}(Z)$ (resp.\ $\cSh_{\Lambda}(Z)$) the full subcategory of objects $F \in \lSh(Z) \ (\text{resp.\ } \cSh(Z))$ satisfying $\MS(F) \subset \Lambda$.
The homotopy category of $\lSh_{\Lambda}(Z)$ (resp.\ $\cSh_{\Lambda}(Z)$) is equivalent to the derived category of quasi-constructible (resp.\ constructible) complexes of sheaves microsupported in $\Lambda$.

Let $\wSh_{\Lambda}(Z)$ be the full dg subcategory of compact objects in $\lSh_{\Lambda}(Z)$ whose objects are called \emph{wrapped constructible sheaves}, which is defined by Nadler (\cite{nadler2016wrapped}).

Let $p \in \Lambda_{\mathrm{reg}}$ be a smooth point of $\Lambda$ and $f$ be a real analytic test function satisfying conditions (i)--(iii) stated just before Definition~\ref{microlocalstalk}. 
The microlocal stalk functor $\phi_{p,f} \colon \lSh_\Lambda(Z) \to \Module(k)$ preserves limits since it is a composite of functors $\Gamma_{\{f \ge 0\}}(\ast)$ and $\Gamma(B;\ast)$, which are both right adjoint functors.
Moreover $\phi_{p,f}$ preserves colimits since 
\begin{equation}
\Gamma_{\{f \ge0\}}(B;F) \simeq \hom(\bC_{\{f \ge 0\}\cap B},F)
\end{equation}
and $\bC_{\{f \ge 0\}\cap B}$ is constructible (see Proposition~\ref{constiswrapped}).
Hence $\phi_{p,f}$ admits a left adjoint functor $\phi^l_{p,f} \colon \Module(k) \to \lSh_\Lambda(Z)$ and $\phi^l_{p,f}$ preserves compact objects.

\begin{definition}
	Let $\Lambda,p$, and $f$ be as above.
	Define the \emph{microlocal skyscraper} $F_{p, f}:=\phi^l_{p,f}(k) \in \wSh_{\Lambda}(Z)$ as the object representing the microlocal stalk functor:
	\begin{equation}
	\phi_{p,f}(F) =
	\Hom(F_{p,f},F) \quad 
	(F \in \lSh_\Lambda(Z)).
	\end{equation}
\end{definition}

Nadler (\cite{nadler2016wrapped}) gives a more geometric characterization of $\wSh_\Lambda(Z)$ by using microlocal skyscrapers.

\begin{lemma}[{\cite[Lemma~3.15]{nadler2016wrapped}}]
	The dg category $\wSh_\Lambda(Z)$ is split-generated by the microlocal skyscrapers $F_{p,f} \ (p \in \Lambda_{\mathrm{reg}})$.
\end{lemma}

\subsection{Preliminaries on CCC}
In this section, we review the coherent-constructible correspondence.
\subsubsection{Formulation}
Let $M$ be a free abelian group of rank $n$ and $N$ the dual abelian group of $M$.
Let $\Sigma$ be a smooth fan in $N_\bR$.
We denote by $X_\Sigma$ the smooth toric variety associated with $\Sigma$ (see, e.g., \cite{CLS}).
We also set 
\begin{equation}\label{eq:lagsubset}
\begin{split}
\widetilde{\Lambda_\Sigma}&:=
\bigcup_{\sigma \in \Sigma}
(\sigma^\perp+M) \times (-\sigma)
\subset T^*M_\bR,\\
\Lambda_\Sigma
&:=
\bigcup_{\sigma \in \Sigma}
p(\sigma^\bot) \times (-\sigma)
\subset T^*(M_\bR/M)
= T^*T^n,
\end{split}
\end{equation}
where $p \colon M_\bR \simeq \bR^n \to M_\bR/M =: T^n$ is the projection.
Note that $\Lambda_\Sigma$ is a conic Lagrangian subvariety in $T^*T^n$.

Denote the dg category of coherent sheaves on $X_\Sigma$ by $\Coh X_\Sigma$ which is localized at quasi-isomorphisms.
We also denote by $\cSh_{\Lambda_{\Sigma}}(T^n)$ the dg category of $\bR$-constructible objects on $T^n$ whose microsupports are contained in $\Lambda_\Sigma$ (which is also localized at quasi-isomorphisms).
When $\Sigma$ is complete, there is a well-defined fully faithful morphism (cf.\ \cite{FLTZ, Tr})
\begin{equation}\label{mor:ccc}
\kappa_\Sigma \colon 
\Coh X_\Sigma \lto \cSh_{\Lambda_{\Sigma}}(T^n).
\end{equation}
We will recall the construction of $\kappa_\Sigma$ in the next section.

Fang--Liu--Treumann--Zaslow formulated the following conjecture, called the \emph{coherent-constructible correspondence}, and its equivariant version, and proved the latter (\cite{FLTZ, FLTZ2}). 
\begin{conjecture}[The coherent-constructible correspondence, CCC]\label{conj:ccc}
	For a smooth complete fan $\Sigma$, 
	the morphism \eqref{mor:ccc} induces a quasi-equivalence of dg categories. 
\end{conjecture}

In some cases, proofs were known before this paper (see also Remark \ref{after}).
\begin{theorem}[\text{\cite{Tr, SS, Kuw}}]\label{thm:kuwagaki}
	If $\Sigma$ is cragged or $\dim\Sigma=2$, the morphism $\kappa_\Sigma$ is a quasi-equivalence.
\end{theorem}\noindent

We conjecture a variant of Conjecture~\ref{conj:ccc} to include non-complete $\Sigma$'s. Note that there exists a fully faithful morphism $\kappa_\Sigma\colon \Coh X_\Sigma\rightarrow \lSh_{\Ls}(T^n)$ by the construction of $\kappa_\Sigma$ even when $\Sigma$ is not complete.
\begin{conjecture}\label{conj:gccc}
	The morphism $\kappa_\Sigma$ induces a quasi-equivalence
	\begin{equation}
	\kappa_\Sigma\colon  \Coh X_\Sigma\xrightarrow{\simeq} \wSh_{\Ls}(T^n).
	\end{equation}
\end{conjecture}

There is also a ``quasi-version'' (or ind-version) of the above conjecture.
\begin{conjecture}\label{conj:qccc}
	The morphism $\kappa_\Sigma$ induces a quasi-equivalence
	\begin{equation}\label{qccceq}
	\kappa_\Sigma\colon \QCoh X_\Sigma\xrightarrow{\simeq} \lSh_\Ls(T^n)
	\end{equation}
\end{conjecture}
\begin{remark}\label{after}
	A proof of Conjecture~\ref{conj:qccc} is sketched in \cite{Vaintrob} and was independently proved in \cite{Kuw2} after the appearance of the first version of this paper.
	By taking compact objects of (\ref{qccceq}), Conjecture~\ref{conj:gccc} is equivalent to Conjecture~\ref{conj:qccc}.
\end{remark}

\begin{proposition}\label{constiswrapped}
	The objects of $\cSh_\Ls(T^n)$ are compact in $\lSh_\Ls(T^n)$.
\end{proposition}
\begin{proof}
	Let $F \in \cSh_\Ls(T^n)$ be a constructible object.
	Take a stratification $\cS$ consisting of connected strata such that $\Lambda_\Sigma \subset \bigsqcup_{S \in \cS} T^*_{S}T^n=:\Lambda_\cS$. Then $\lSh_\Ls(T^n)$ is equivalent to a dg category of dg modules over a quiver path algebra (e.g., \cite{STZ}).
	Since $F$ is locally constant and has a finite-dimensional stalk on each stratum $S$, it is a perfect object in $\lSh_{\Lambda_\cS}(T^n)$.
	Since $\Ls \subset \Lambda_\cS$, the result follows.
\end{proof}

\begin{corollary}\label{gccctoccc}
	Conjecture~\ref{conj:ccc} for $\Sigma$ follows from Conjecture~\ref{conj:gccc} for $\Sigma$.
\end{corollary}
\begin{proof}
	This follows from the fact that the functor $\ks|_{\Coh}$ is a fully faithful functor into $\cSh\subset \wSh$ (\cite{Tr}). 
\end{proof}
In particular, wrapped constructible sheaves along $\Ls$ coincide with constructible sheaves along $\Ls$ for complete $\Sigma$ under Conjecture~\ref{conj:qccc}. In Appendix~\ref{sec:surface} of this paper, we give a proof of this coincidence for $\dim\Sigma=2$ without assuming Conjecture~\ref{conj:qccc}.

\subsubsection{Construction of $\kappa_\Sigma$}
First we recall the construction of $\kappa_\Sigma$ due to Fang--Liu--Treumann--Zaslow (\cite{FLTZ, FLTZ2}).

For $\sigma\in \Sigma$, we set
\begin{equation}
\begin{split}
\Theta(\sigma)&:=p_!\bC_{\Int(\sigma^\vee)}[n]\in \QCoh X_\Sigma, \\
\Theta'(\sigma)&:=i_{{U_\sigma}*}\cO_{\sigma}\in \lSh_{\Ls}(T^n),
\end{split}
\end{equation}
where $p\colon M_\bR\rightarrow T^n$ is the projection, $\bC_{\Int(\sigma^\vee)}$ is the constant sheaf on $\Int(\sigma^\vee)\subset M_\bR$ the interior of the polar dual of $\sigma$, $i_{U_\sigma}\colon U_\sigma\hookrightarrow X_\Sigma$ is the affine toric coordinate of $X_\Sigma$ corresponding to $\sigma$, and $\cO_{\sigma}$ is the structure sheaf of $U_\sigma$.
Note that our definition of $\Theta(\sigma)$ is slightly different from that of \cite[Definition~3.1]{FLTZ}.

For $\sigma, \tau\in \Sigma$, there exists an isomorphism
\begin{equation}
\begin{split}
H^i\hom^\bullet_{\lSh_\Ls(T^n)}(\Theta(\sigma),\Theta(\tau))\simeq \begin{cases}
\bC[\sigma^\vee\cap M] &\text{ when $\sigma\supset \tau$ and $i=0$,}\\
0  &\text{ otherwise}.
\end{cases} 
\end{split} 
\end{equation} 
For $m\in \sigma^\vee\cap M$, the corresponding morphism in the left-hand side of the above is given by 
\begin{equation}
\Theta(\sigma):=p_!\bC_{\Int(\sigma^\vee)}[n] \simeq p_!\bC_{\Int(\sigma^\vee+m)}[n]\xrightarrow{p_!(\iota_{\sigma^\vee+m, \tau^\vee})[n]}p_!\bC_{\Int(\tau^\vee)}[n]=:\Theta(\tau),
\end{equation}
where $\iota_{m+\sigma^\vee, \tau^\vee}$ is the canonical morphism $\bC_{\Int(\sigma^\vee+m)}\rightarrow \bC_{\Int(\tau^\vee+m)}$ induced by the inclusion $\sigma^\vee+m\hookrightarrow \tau^\vee$.

Similarly, for $\sigma, \tau\in \Sigma$, there exists an isomorphism
\begin{equation}
\begin{split}
H^i\hom^\bullet_{\QCoh X_\Sigma}(\Theta'(\sigma),\Theta'(\tau))\simeq \begin{cases}
\bC[\sigma^\vee\cap M] &\text{ when $\sigma\supset \tau$ and $i=0$,}\\
0  &\text{ otherwise}.
\end{cases} 
\end{split} 
\end{equation} 
For $m\in \sigma^\vee\cap M$, the corresponding morphism in the left-hand side of the above is given by the restriction morphism $\theta'_m\colon i_{U_\sigma*}\cO_\sigma\rightarrow i_{U_\tau*}\cO_\tau$.

We denote by $\Gamma(\Sigma)$ the dg category whose set of objects is $\Sigma$ and hom spaces are defined by
\begin{equation}
\begin{split}
\hom^i_{\Gamma(\Sigma)}(\Theta(\sigma),\Theta(\tau)):= \begin{cases}
\bC[\sigma^\vee\cap M] &\text{ when $\sigma\supset \tau$ and $i=0$,}\\
0  &\text{ otherwise},
\end{cases}
\end{split}
\end{equation}
with trivial differentials. Then there exist quasi-equivalent dg functors $\Gamma(\Sigma)\rightarrow \Theta_\Sigma$ and $\Gamma(\Sigma)\rightarrow \Theta'_\Sigma$, where $\Theta_\Sigma$ (resp. $\Theta'_\Sigma$) is the full subcategory of $\QCoh X_\Sigma$ (resp. $\lSh_\Ls(T^n)$) spanned by $\lc \Theta(\sigma)\relmid \sigma\in \Sigma\rc$ (resp. $\lc \Theta'(\sigma)\relmid \sigma\in \Sigma\rc$). 
Hence we have a quasi-equivalent morphism $\kappa_\Sigma\colon \Perf\Theta_\Sigma\rightarrow \Perf\Theta'_\Sigma$, where $\Perf\cA$ denotes the full subcategory of the category of dg modules over a dg category $\cA$ spanned by perfect modules.

By using $\mathrm{\check{C}ech}$ resolution, it follows that $\Coh X_\Sigma\subset \Perf\Theta_\Sigma$. Hence $\kappa_\Sigma$ is restricted to $\kappa_\Sigma\colon \Coh X_\Sigma\rightarrow \lSh_\Ls(T^n)$. We denote this morphism by $\kappa_\Sigma$. Since $\lSh_\Ls(T^n)$ is cocomplete and $\kappa_\Sigma$ is finite-colimit preserving (since it is induced from an equivalence), there also exists $\kappa_\Sigma\colon \QCoh X_\Sigma\rightarrow \lSh_\Ls(T^n)$ by abuse of notation.

\subsubsection{Properties of the morphism $\kappa_\Sigma$}
In this section, we list three properties of $\kappa_\Sigma$, which we will use later in this paper.

The first one is the naturality. 
Let $\Sigma'$ be a subfan of $\Sigma$. 
Then we have $\Lambda_{\Sigma'}\subset \Ls$, accordingly $\lSh_{\Lambda_{\Sigma'}}(T^n)\subset \lSh_{\Ls}(T^n)$. 
We also have $\Theta_{\Sigma'}\subset \Theta_{\Sigma}$, $\Theta'_{\Sigma'}\subset \Theta'_\Sigma$ and $\Gamma({\Sigma'})\subset\Gamma(\Sigma)$. 
Hence, by definition, $\kappa_{\Sigma'}$ can be obtained as the restriction of $\kappa_{\Sigma}$. 
In other words, we have the following proposition.
\begin{proposition}[\cite{FLTZ}]\label{naturality}
	\begin{equation}
	\kappa_{\Sigma}(\Coh X_{\Sigma'})\subset \lSh_{\Lambda_{\Sigma'}}(T^n).
	\end{equation} 
\end{proposition}

The second property is the monoidality. Let $m\colon T^n\times T^n\rightarrow T^n$ be the multiplication. We define the convolution product $\lSh(T^n)\times \lSh(T^n)\rightarrow \lSh(T^n)$ by $E* F:=m_!(E\boxtimes F)$. This operation defines a monoidal structure on $\lSh(T^n)$.

\begin{proposition}[\cite{FLTZ}]\label{prp:monoidality} For any $\cE, \cF\in \Coh X_\Sigma$, we have
	\begin{equation}
	\kappa_\Sigma(\cE \otimes \cF)=\kappa_{\Sigma}(\cE)*\kappa_{\Sigma}(\cF).
	\end{equation} 
\end{proposition}

The third property is naturality for direct products. Let $\Sigma$ and $\Sigma'$ be smooth fans. The product of two fans is defined as
\begin{equation}
\Sigma\times \Sigma':=\lc \sigma\times \sigma'\relmid \sigma\in \Sigma, \sigma'\in \Sigma'\rc,
\end{equation}
which associates the product of toric varieties $X_\Sigma\times X_{\Sigma'}$ (e.g., \cite{CLS}). It is clear from the definitions that $\Gamma(\Sigma\times \Sigma')=\Gamma(\Sigma)\times \Sigma(\Sigma')$, $\Theta(\sigma)\boxtimes \Theta(\sigma')=\Theta(\sigma\times \sigma')$, and $\Theta'(\sigma)\boxtimes \Theta'(\sigma')=\Theta'(\sigma\times \sigma')$. As a consequence, we have the following.
\begin{proposition}\label{prp:product}
	In the above notation, we have
	\begin{equation}
	\kappa_{\Sigma\times \Sigma'}\simeq \kappa_\Sigma\times \kappa_{\Sigma'}.
	\end{equation}
\end{proposition}

\subsubsection{Mirrors of points in toric varieties}
In this section, we describe the image of the structure sheaves of points in toric varieties under the functor $\kappa_\Sigma$. This subject should be strongly related to the T-duality picture of mirror symmetry (\cite{FLTZ2, SYZ}).

For a point $x\in X_\Sigma$, take the smallest open toric subvariety $U_\sigma$ which contains $x$. 
Let $\check{\rho}_1,\dots, \check{\rho}_r$ be the 1-dimensional faces (rays) of $\sigma^\vee$ the polar dual of $\sigma$. Then the primitive generators of the $\check{\rho}_i$'s can be extended to a basis $m_1,\dots, m_n$ of $M$, where $m_1,\dots, m_r$ correspond to $\check{\rho}_1,\dots, \check{\rho}_r$, respectively. Then the open toric subvariety $U_\sigma$ can be identified as
\begin{equation}
U_\sigma=\Spec\bC[X_1,\dots, X_r, X_{r+1}^\pm,\dots , X_n^\pm]\simeq \bC^r\times (\bC^*)^{n-r}.
\end{equation} 
In the following, we will view the right-hand side of the above isomorphism as the coordinate of $U_\sigma$.
Since $U_\sigma$ is the smallest open toric variety containing $x$, we have $x=(0,\dots, 0, x_{r+1},\dots, x_n )$ and $x_i\neq 0$ for $i\geq r+1$.

Since $m_1,\dots, m_n$ is a basis of $M$, we can factorize $M_\bR$ and $T^n$ as $M_\bR\simeq \bR^r\times \bR^{n-r}$ and $T^n\simeq T^r\times T^{n-r}$. Let $\bC_{(X_1,\dots, X_r)}$ be the proper push-forward to $T^r$ of the constant sheaf on $\bR^r$ supported on $(0,1]^r$. Further, let $\bC_{(X_{r+1}-x_{r+1},\dots, X_{n}-x_{n})}$ be the locally constant sheaf on $T^{n-r}$ whose monodromy along $m_i$ corresponds to the multiplication by $x_i$. Here we take a set of generators of the fundamental group $\pi_1(T^{n-r})$ as $m_{r+1},\dots, m_n$. We view the exterior product $P_x:=\bC_{(X_1,\dots, X_r)}\boxtimes \bC_{(X_{r+1}-x_{r+1},\dots, X_{n}-x_{n})}$ as a sheaf on $T^n=M_\bR/M$.

\begin{proposition}\label{points} Under the above notation, we have
	\begin{equation}
	\kappa_{\Sigma}(\cO_x)\simeq P_x.
	\end{equation} 
\end{proposition}
Under the T-duality picture, the big algebraic torus $(\bC^*)^n$ in $X_\Sigma$ should be viewed as the non-degenerate torus fibration structure (\cite{FLTZ2}) and the structure sheaves on these points correspond to the local systems on $T^n$ which are expected to be equivalent to the torus fibers with a $U(1)$-local system in the Fukaya category of $T^*T^n$. On the other hand, points in toric divisors correspond to {\em degenerate local systems} which are ``local systems with $0$ or $\infty$ monodromies'' and such sheaves have non-compact Lagrangians as these microsupports. This is expected to have some relation to other descriptions of SYZ mirror to non-flat torus fibrations (e.g., \cite{AAK}).  

\begin{proof}
	On $U_\sigma$, we have the Koszul resolution of the skyscraper sheaf $\cO_x$ on $x$:
	\begin{equation}
	0\leftarrow \cO_x\leftarrow \cO_\Us\leftarrow \cO_{\Us}^{\oplus n}\leftarrow \bigwedge^2(\cO_\Us^{\oplus n})\leftarrow \bigwedge^3(\cO_\Us^{\oplus n}) \leftarrow \cdots \leftarrow \cO_\Us \leftarrow 0.
	\end{equation} 
	As $R$-modules ($R:=\bC[X_1,\dots, X_r, X_{r+1}^\pm,\dots , X_n^\pm]$), this resolution is 
	\begin{equation}
	0\leftarrow R/(X_1,\dots, X_r, X_{r+1}-x_{r+1},\dots, X_n-x_n)\leftarrow R\leftarrow R^{\oplus n}\leftarrow \cdots \leftarrow R\leftarrow 0
	\end{equation} 
	and the morphism is the internal differential by $\iota_{\sum{X_i-x_i}}$. 
	This resolution is the exterior tensor product of the resolutions 
	\begin{equation}
	0\leftarrow \bC[X_i]/(X_i-x_i)\leftarrow \bC[X_i]\leftarrow \bC[X_i]\leftarrow 0
	\end{equation} 
	where the morphism $\bC[X_i]\rightarrow \bC[X_i]$ is given by the multiplication of $X_i-x_i$. 
	Let $\Sigma_{\bA^1}$ be the fan of the affine toric variety $\bA^1=\bC$. Then we have
	\begin{equation}
	\kappa_{\Sigma_{\bA^1}}(\bC[X_i]/(X_i-x_i))= \bC_{(X_i-x_i)}.
	\end{equation} 
	The Proposition~\ref{prp:product} implies $\kappa_{\Sigma_{\bA^n}}=\kappa_{\Sigma_{\bA^1}}^{\times n}$, which gives the conclusion.
\end{proof}

\subsection{Seidel's categorical localization of $\Coh X$}

\begin{proposition}[{\cite[p.\ 87 Complements of divisors]{Seidelnat}}]\label{prp:localization}
	Let $X$ be a smooth toric variety over $\bC$, $D \subset X$ be a toric divisor, and $U:=X \setminus D$ be its complement.
	\begin{enumerate}
		\item Let $s$ be the canonical section of the line bundle $\cO(D)$.
		Then for any $\cF \in \Coh X$, one has
		\begin{equation}
		j_* \cF|_U
		\simeq 
		\colim_l \cF \otimes \cO(D)^{\otimes l},
		\end{equation}
		where $j \colon U \hookrightarrow X$ is the inclusion and the colimit is formed with respect to multiplication with $s$.
		\item Set
		\begin{equation}
		\Coh_D X
		:=
		\{\cF \in \Coh X \mid \Supp \cF \subset D\}.
		\end{equation}
		Then the restriction functor 
		\begin{equation}
		\Coh X/\Coh_D X
		\lto \Coh U
		\end{equation}
		is a quasi-equivalence of dg categories. 
	\end{enumerate}
\end{proposition}

\begin{proof} First, we assume that $X$ is quasi-projective.
	
	(i) By \cite{Seidelnat}, we have 
	\begin{equation}
	\colim_l \Gamma(V;\cF \otimes \cO(D)^{\otimes l})
	\simeq 
	\Gamma(U \cap V;\cF|_U)
	\end{equation}
	for any open set $V$ of $X$ and obtain the desired isomorphism.
	
	\noindent (ii) This is given in \cite{Seidelnat} for quasi-projective cases.
	
	Finally, we let $X$ be a (not necessarily quasi-projective) toric variety. In our case, $D$ is not arbitrary but a toric divisor. Hence the star neighborhood of the corresponding ray to $D$, i.e., the subfan consisting of the faces of the maximal-dimensional cones containing the ray to $D$ in $\Sigma$ gives a open quasi-projective toric subvariety $X'$ and the complement of the star neighborhood never intersects $D$. Hence, we can apply the above argument to $X'$ and glue it to the complement of $X'$.
\end{proof}

\section{Main theorem}\label{sec:main}

Let $\Sigma$ be a smooth fan in $N_\bR$ and $\rho \in \Sigma$ be a 1-dimensional cone in $\Sigma$. 
Denote the toric divisor defined as the closure of the $T$-orbit $T_\rho$ associated with $\rho$ by $D_\rho$.
Set $
\Sigma^\rho_c:=
\Sigma
\setminus 
\{ \tau \in \Sigma \mid \rho \preceq \tau \}
$, the complement of the star neighborhood of $\rho$, and denote the open inclusion by $i\colon X_{\Sigma_c^\rho}=X_\Sigma\bs D_\rho\hookrightarrow X_\Sigma$.

Further, let $\cB_\rho$ be the subcategory of $\wSh_\Ls(T^n)$ split-generated by microlocal skyscraper sheaves $\lc F_{p,f}\relmid p\in (\Ls\backslash \Lambda_{\Sigma^\rho_c})_\mathrm{reg} \rc$ and $\Coh_{D_\rho} X_\Sigma$ be the full dg subcategory of $\Coh X_\Sigma$ consisting of objects supported in $D_\rho$. 
\begin{theorem}\label{main}Assuming Conjecture~\ref{conj:qccc} for $\Sigma$, there is a quasi-equivalence
	\begin{equation}
	\wSh_{\Lambda_{\Sigma^\rho_c}}(T^n)\simeq \wSh_\Ls(T^n)/\cB_\rho
	\end{equation}
	such that the induced morphism $\wSh_\Ls(T^n)\rightarrow \wSh_{\Lambda_{\Sigma^\rho_c}}(T^n)$ fits into the diagram
	\begin{equation}
	\xymatrix@C=40pt{
		\Coh X_\Sigma \ar[r]^-{\kappa_\Sigma} \ar[d]_{i^*} &\ar[d] \wSh_\Ls(T^n)\\
		\Coh X_{\Sigma^\rho_c} \ar[r]_-{\kappa_{\Sigma^\rho_c}} & \wSh_{\Lambda_{\Sigma_c^\rho}}(T^n). 
	}
	\end{equation}
\end{theorem}
A proof of this theorem is given at the end of this section. Throughout this section, we assume Conjecture~\ref{conj:qccc} for $\Sigma$.

Let $\cN$ be a dg subcategory $\kappa_\Sigma (\Coh_{D_\rho} X_\Sigma) \subset \wSh_{\Lambda_\Sigma}(T^n)$.
Set $D:=D_\rho$ for simplicity. 
Define $F_D:=\kappa_\Sigma(\cO(D))$ and 
\begin{equation}
\beta(F):=
\colim_l F * (F_D)^{*l}
\in \lSh_{\Lambda_{\Sigma}}(T^n)
\end{equation}
for $F \in \wSh_{\Lambda_{\Sigma}}(T^n)$. Assuming Conjecture~\ref{conj:qccc}, by Proposition~\ref{prp:localization}, $\beta\circ \kappa_\Sigma=\kappa_{\Sigma}\circ j_*j^*$, where $j\colon X_{\Sigma^\rho_c}\hookrightarrow X_\Sigma$ is the inclusion.
In other words,
\begin{equation}
\beta(F)\simeq F*\kappa_\Sigma(\cO_{X_{\Sigma_c^\rho}}).
\end{equation}

\begin{lemma}\label{lem:conv}
	\begin{enumerate}
		\item $F * F' \in \cN$ if $F' \in \cN$.
		\item $\Cone((F_D)^{*l_1} \to (F_D)^{*l_2}) \in \cN\ (l_2 >l_1)$.
		\item $\beta(F) \simeq 0$ if $F \in \cN$.
	\end{enumerate}
\end{lemma}
\begin{proof}
	These are clear from the definition of $\cN$, Conjecture~\ref{conj:qccc}, the monoidality (Proposition~\ref{prp:monoidality}) and the cocontinuity of $\kappa_\Sigma$.
\end{proof}

Consider the following commutative diagram of dg functors
\begin{equation}
\xymatrix{
	\cC \ar[r]^-{\beta} \ar[d] & \Ind(\cC)=\lSh_{\Lambda_{\Sigma}}(T^n) \\
	\cC/\cN \ar[r]_-{\beta_\cN} & \lSh_{\Lambda_{\Sigma_c^\rho}}(T^n), \ar@{^{(}->}[u]_-{i}
}
\end{equation}
where $\Ind$ is the ind-completion (cf.\ \cite{MR1827714, MR2182076}).
The bottom arrow is defined by using Lemma~\ref{lem:conv}(iii).
By the cocompleteness of $\lSh_{\Lambda_c^\rho}(T^n)$, $\beta_\cN$ extends to $\Ind(\cC/\cN)$:
\begin{equation}
\xymatrix@C=15pt{
	\cC/\cN \ar[rr]^-{\beta_\cN} \ar@{^{(}->}[rd]& & \lSh_{\Lambda_{\Sigma_c^\rho}}(T^n) \\
	& \Ind(\cC/\cN). \ar@{-->}[ru]_-{\Phi}
}
\end{equation}
There exists a canonical candidate for a quasi-inverse of $\Phi$. Let $\tl q \colon \Ind(\cC) \to \Ind(\cC/\cN)$ be the morphism induced by $\cC \to \cC/\cN$.
Define $\Psi$ to be the composite of $i \colon \lSh_{\Lambda_{\Sigma_c^\rho}}(T^n) \hookrightarrow \lSh_{\Lambda_{\Sigma}}(T^n) \simeq \Ind(\cC)$ and $\tl q$:
\begin{equation}
\Psi:=\tl q \circ i \colon \lSh_{\Lambda_{\Sigma_c^\rho}}(T^n) \hookrightarrow \lSh_{\Lambda_{\Sigma}}(T^n) \simeq \Ind(\cC) \lto \Ind(\cC/\cN).
\end{equation}
\begin{lemma}\label{wrappedequivalence}
	The functor $\Psi$ is a quasi-inverse to $\Phi$. Hence, $\Ind(\cC/\cN)\simeq \lSh_{\Lambda_c^\rho}(T^n)$. In particular, there is a quasi-equivalence
	\begin{equation}
	\beta_\cN \colon \wSh_{\Lambda_\Sigma}(T^n)/\cN \xrightarrow{\simeq} \wSh_{\Lambda_{\Sigma^\rho_c}}(T^n) \subset \lSh_{\Lambda_{\Sigma^\rho_c}}(T^n)
	\end{equation}
	by taking compact objects.
\end{lemma}

\begin{proof}
	Denote by $[F]$ the image of $F \in \cC$ in $\cC/\cN$.
	Any object of $\Ind(\cC/\cN)$ is of the form of $\indlim_i [F_i]  \ (F_i \in \cC)$, where $\indlim$ is a formal limit (cf.\ \cite{MR2182076}).
	By construction, the functors $\Phi$ and $\tl q$ are described as 
	\begin{align}
	\Phi \colon \indlim_{i} [F_i] & \longmapsto
	\colim_i \colim_{l} F_i * (F_D)^{*l}, \\
	\tl q \colon \indlim_{i} F_i & \longmapsto
	\indlim_{i} [F_i].
	\end{align}
	Moreover, the equivalence $\Ind(\cC) \xrightarrow{\simeq} \lSh_{\Lambda_{\Sigma}}$ can be described as 
	\begin{align}
	\indlim_{i} F_i \longmapsto
	\colim_{i} F_i.
	\end{align}
	Hence, the functor $\Psi \circ \Phi \colon \Ind(\cC/\cN) \to \Ind(\cC/\cN)$ is  
	\begin{equation}
	\Psi \circ \Phi \colon 
	\indlim_{i} [F_i] \longmapsto 
	\indlim_{i} \indlim_{l} [ F_i * (F_D)^{*l}]. 
	\end{equation}
	By Lemma~\ref{lem:conv}(i) and (ii), $F_i \to F_i * (F_D)^{*l}$ is an isomorphism in $\cC/\cN$. 
	Therefore the image of $\indlim_{i} [F_i]$ under $\Psi \circ \Phi$ is isomorphic to itself and $\Psi \circ \Phi \simeq \id$.
	Next we consider $\Phi \circ \Psi$.
	Let $G \in \lSh_{\Lambda_{\Sigma_c^\rho}}$ and write $G \simeq \colim_i F_i \ (F_i \in \cC)$.
	By construction, 
	\begin{equation}
	\Phi \Psi(G) =\colim_i \beta(F_i) =
	\colim_i \colim_l F_i * (F_D)^{*l}.
	\end{equation}
	Consider a distinguished triangle 
	\begin{equation}
	F_i \lto F_i * (F_D)^{*l} \lto \Cone(F_i \to F_i * (F_D)^{*l}) 
	\overset{+1}{\lto}.
	\end{equation}
	By Lemma~\ref{lem:conv}(i) and (ii), $\Cone(F_i \to F_i * (F_D)^{*l}) \in \cN$.
	Taking colimits, we obtain a distinguished triangle
	\begin{equation}
	\colim_i F_i \lto 
	\colim_i \beta(F_i) \lto 
	\colim_i \colim_{l} \Cone(F_i \to F_i * (F_D)^{*l}) 
	\overset{+1}{\lto}.
	\end{equation}
	Here $\colim_i F_i \simeq G \in \lSh_{\Lambda_{\Sigma_c^\rho}}(T^n)$ and $\beta(F_i) \in \lSh_{\Lambda_{\Sigma_c^\rho}}(T^n)$ by Proposition~\ref{naturality}.
	By the distinguished triangle and the triangular inequality for microsupports, we get 
	\begin{equation}
	\colim_i \colim_{l} \Cone(F_i \to F_i * (F_D)^{*l}) \in \lSh_{\Lambda_{\Sigma_c^\rho}}(T^n).
	\end{equation}

	\begin{lemma}\label{lem:vanish}We have
		\begin{equation}
		^\perp\lSh_{\Lambda_{\Sigma^\rho_c}}(T^n)\cap \wSh_\Ls(T^n)=\cN.
		\end{equation}
		As a corollary,
		\begin{equation}
		\Hom_{\lSh_{\Lambda_{\Sigma}}}(\overline{\cN}^c,\lSh_{\Lambda_{\Sigma^\rho_c}})=0
		\end{equation}
		where $\overline{\cN}^c$ is the filtered cocomplete closure of $\cN$. 
	\end{lemma}
	
	\begin{proof}
		First, we show that $^\perp\lSh_{\Lambda_{\Sigma^\rho_c}}(T^n)\cap \wSh_\Ls(T^n)\supset\cN$. It is enough to show that
		\begin{equation}
		\Hom_{\Coh(X_\Sigma)}(\cF,\kappa_\Sigma^{-1}H)=0
		\end{equation}
		for any $\cF \in \Coh_{D_\rho} X_\Sigma$ and $H \in \lSh_{\Lambda_{\Sigma^\rho_c}}(T^n)$.
		Set $\cG:=\kappa_\Sigma^{-1}H$.
		Since $\cF$ is coherent, we have
		\begin{equation}
		\cHom(\cF,\cG)_x
		\simeq 
		\Hom_{\cO_x}(\cF_x,\cG_x),
		\end{equation}
		where $\cHom$ in the left-hand side denotes the internal Hom in the dg category of quasi-coherent sheaves. 
		For $x \not\in D_\rho$, $\cF_x \simeq 0$ and the right-hand side vanishes.
		It is enough to show that $\hom_{\cO_x}^j(\cG_x,\cO_x) \simeq 0$ for $x \in D_\rho$ since $\Coh(\{x\}) \simeq \langle \cO_x \rangle \simeq \mathbf{Vect}_\bC$.
		Denote by $i_x \colon \{x\} \hookrightarrow X$ the inclusion.
		Then we have the isomorphisms
		\begin{align*}
		\cHom_{\cO_x}(\cG_x,\cO_x) 
		& \simeq 
		\Hom_{\cO_x}(i_x^*\cG,\cO_x) \\
		& \simeq 
		\Hom_{\cO_{X_\Sigma}}(\cG,\cO_x) \\
		& \simeq 
		\Hom_{\lSh}(H,\ks(\cO_x)) .
		\end{align*}
		Fix an integral basis of $N$ including the primitive vector of $\rho$. Then we have the dual integral basis of $M$, which has $v_\rho$ the dual to $\rho$.
		As we described in Proposition~\ref{points}, $\kappa(\cO_x)$ is of the form
		\begin{equation}
		\tl{p}_!(\cL \boxtimes \bC_{(0,1]^r}),
		\end{equation}
		where the second component of the exterior tensor is the direction $v_\rho$, the first part is spanned by the others, $\cL$ is a local system on $T^{n-r}$ and $\tl p \colon T^{n-r} \times \bR^r \to T^n$ is the quotient map. 
		Then we obtain an isomorphism
		\begin{equation}
		\Hom(H,\kappa(\cO_x)) 
		\simeq 
		\Hom(\tl{p}^{-1}H,\cL \boxtimes \bC_{(0,1]^r}).
		\end{equation}
		Here $\MS(\tl{p}^{-1}H) \subset (T^{n-r} \times \bR^r) \times \bigcup_{\sigma \in \Sigma_c^\rho} (-\sigma)$ and $\MS(\cL \boxtimes \bC_{(0,1]^r}) \subset (T^{n-r} \times \bR^r) \times \bigcup_{\rho \preceq \tau} (-\tau)$.
		Since a closed cone $\tau$ satisfying $\rho \preceq \tau$ does not contain $-\rho$, by Proposition~\ref{prp:SSHom} $\MS(\cHom(\tl{p}^{-1}H,\cL \boxtimes \bC_{(0,1]^r}))$ does not contain $\rho$.
		Thus we can apply the microlocal Morse lemma (Proposition~\ref{MML}) to the function $\psi(t,x):=\langle v'_\rho,x \rangle \ ((t,x) \in T^{n-r} \times \bR^r)$, where $v'_\rho$ is the vector in $\bR^r$ corresponding to $v_\rho$.
		Noticing that $\cHom(\tl{p}^{-1}H,\cL \boxtimes \bC_{(0,1]^r})$ has compact support, we obtain $\Hom(H,\kappa(\cO_x)) \simeq \Hom(\tl{p}^{-1}H,\cL \boxtimes \bC_{(0,1]^r})=0$. As a corollary, the second line in the statement follows.
		
		Conversely, for $E\in {}^\perp\lSh_{\Lambda_{\Sigma^\rho_c}}(T^n)\cap \wSh_\Ls(T^n)$, we want to see that $\ks^{-1}(E)_x\simeq 0$ for $x\in X_{\Sigma^\rho_c}$. This is implied by
		\begin{equation}
		\begin{split}
		0\simeq \Hom_{\lSh}(\ks(\cO_x), E)&\simeq \Hom_{X_\Sigma}(\cO_x, \ks^{-1}(E))\\
		&\simeq \Hom_{X_\Sigma}(\ks^{-1}(E), \cO_x)^\vee \\
		&\simeq \Hom_{\cO_x}(\ks^{-1}(E)_x, \cO_x)^\vee.
		\end{split} 
		\end{equation} 
	\end{proof}

	By Lemma~\ref{lem:vanish}, 
	\begin{align*}
	& \Hom \left( \colim_i \colim_{l} \Cone(F_i \to F_i * (F_D)^{*l}),H \right) \simeq 0
	\end{align*}
	for any $H \in \lSh_{\Lambda_{\Sigma_c^\rho}}(T^n)$.
	Therefore we have $\colim_i \colim_{l} \Cone(F_i \to F_i * (F_D)^{*l}) \simeq 0$ and an isomorphism
	\begin{equation}
	G \simto \colim_i \beta(F_i). 
	\end{equation}
	This shows $\Phi \circ \Psi \simeq \id$.
\end{proof}

\begin{lemma}\label{divisorequivalence}
	There is a quasi-equivalence
	\begin{equation}
	\Coh_{D_\rho}(X_\Sigma)\xrightarrow{\simeq} \cN\simeq \cB_\rho
	\end{equation}
	induced by the restriction of $\kappa_\Sigma$.
\end{lemma}
\begin{proof}
	For a microlocal skyscraper sheaf $F_{p,f}\in \wSh_\Ls(T^n)$ at $p\in (\Ls\backslash \Lambda_{\Sigma_c^\rho})_\mathrm{reg}$, this is contained in $\cN$ by Lemma~\ref{lem:vanish}. 
	Suppose that $F\in \overline{N_\rho}^c$ satisfies $\Hom(F_{p,f},F)=0$ for any $p\in (\Ls\backslash \Lambda_{\Sigma^\rho_c})_\mathrm{reg}$ and $f$. 
	Then $\MS(F)\subset \Lambda_{\Sigma_c^\rho}$, i.e., $F\in \lSh_{\Lambda_{\Sigma_c^\rho}}(T^n)$. 
	By Lemma~\ref{lem:vanish}, we conclude that $F=0$. Hence, $\lc F_{p,f}\relmid p\in (\Ls\backslash \Lambda_{\Sigma^\rho_c})_\mathrm{reg}\rc$ compactly generates $\overline{\cN}^c$. 
	Since $\cN$ is a pre-triangulated dg category, $\lc F_{p,f}\relmid p\in (\Ls\backslash \Lambda_{\Sigma^\rho_c})_\mathrm{reg}\rc$ split-generates $\cN$ (\cite[Lemma~4.4.5]{Neeman}).
\end{proof}

\begin{proof}[Proof of Theorem~\ref{main}]
	The desired equivalence and the diagram simultaneously follow from Proposition~\ref{prp:localization} and Lemmas~\ref{wrappedequivalence} and \ref{divisorequivalence}.
\end{proof}

\appendix
\section{Wrapped constructible sheaves are constructible sheaves for mirrors to complete toric surfaces}\label{sec:surface}
In this appendix, we prove that wrapped constructible sheaves coincide with constructible sheaves for $\dim \Sigma=2$ without assuming Conjecture~\ref{conj:qccc}. We consider that such a proof is useful for further developments of the theory of wrapped constructible sheaves and have decided to place it here.

\begin{theorem}\label{compactness}
	Assume $\Sigma$ is a smooth complete fan and $\dim \Sigma=2$. Then the wrapped constructible sheaves along $\Ls$ are constructible.
\end{theorem}
This theorem is concluded after the propositions in this section and Proposition~\ref{constiswrapped}. We also place a proof of Conjecture~\ref{conj:gccc} for $\dim\Sigma=2$ as a corollary at the end of this section.
\smallskip

Let $\cS$ be the coarsest stratification of $T^2$ such that $\Ls\subset \Lambda_\cS:= \bigcup_{S\in \cS}T^*_ST^2$. 
Take $F\in \lSh_\Ls(T^2)$. Let $x\in M_\bR$ be a point over $[x]\in T^2$ and define $C_x\subset M_\bR$ by 
\begin{equation}
C_x:=\bigcap_{\substack{x\in H_{\rho\geq c}\\ \rho\in \Sigma(1), c\in\bZ}}H_{\rho\geq c},
\end{equation} 
where
\begin{equation}
H_{\rho\geq c}:=\lc m\in M_\bR\relmid \la m, \rho\ra \geq c\rc.
\end{equation}
We say that $C_x$ is the closure of $x$ with respect to $\Ls$.

\begin{proposition}
	If $\Hom(F,p_!\bC_{C_x})$ is perfect, $F_{[x]}$ is also perfect.
\end{proposition}
\begin{proof}
	We denote the inclusion by $\iota_{C_x}\colon C_x\hookrightarrow M_\bR$.
	Note that the finite-dimensionality of $\Hom(F, p_*\iota_{C_x*}\bC_{C_x})$ is equivalent to that of $\Hom(\bC_{\Int(C_x)},\bD\iota_{C_x}^{-1}p^{-1}F)=\Gamma(\Int(C_x), \bD \iota_{C_x}^{-1}p^{-1}F)$. We define a continuous family of polyhedral open sets $\{\frakC_t\}_{t\in (0,1)}$ in $\Int(C_x)$ by
	\begin{equation}
	\frakC_t:=x+t(\Int(C_x)-x)
	\end{equation} 
	where the multiplication and the sum are taken with respect to the structure of the vector space of $M_\bR$. This family satisfies
	\begin{enumerate}
		\item $\frakC_t$ is congruent to $\Int(C_x)$ for any $t\in (0,1)$;
		\item $\bigcup_{t\in (0,1)}\frakC_t=\Int(C_x)$;
		\item $\{x\}=\bigcap_{t\in (0,1)}\frakC_t$, and 
		\item the conditions (i) and (ii) of Proposition~\ref{nonchara}. 
	\end{enumerate} 
	If the conormal cone of a face of $\frakC_t$ intersects $-\widetilde{\Ls}$ on $y\in M_\bR$, 
	then $y\in \rho^\perp$ for some $\rho\in \Sigma(1)$. 
	This contradicts the minimality of $C_x$. 
	Hence, Proposition~\ref{nonchara}(iii) is also satisfied and we have
	\begin{equation}
	\Gamma(\Int(C), \bD \iota_{C_x}^{-1}p^{-1}F) \simeq \Gamma(\frakC_t, \bD\iota_{C_x}^{-1}p^{-1}F)
	\end{equation} 
	for any $t\in (0,1)$.
	By the constructibility of $\bD(F)$, for a sufficiently small $t$, we have 
	\begin{equation}
	\Gamma(\frakC_t, \bD\iota_{C_x}^{-1}p^{-1}F)\simeq (\bD\iota_{C_x}^{-1}p^{-1}F)_x\simeq (\bD F)_x.
	\end{equation} 
	Hence, we can conclude that $\Hom(\bC_{\Int(C_x)},\bD\iota_{C_x}^{-1}p^{-1}F)$ is isomorphic to $(\bD F)_x$, hence $F_x$ is finite-dimensional.
\end{proof}

Hence, it suffices to show the following to prove Theorem~\ref{compactness}.
\begin{proposition}\label{keyprop}
	Assume that $F$ is a compact object of $\lSh_\Ls(T^2)$. Let $x\in M_\bR$ be a point and $C_x$ be the closure of $x$ with respect to $\Ls$. If $\Sigma$ is complete, $\Hom(F,p_!\bC_x)$ is perfect. 
\end{proposition}
First, we will prove the following lemma.
\begin{lemma}\label{okcase}
	Assume that $F$ is a compact object of $\lSh_\Ls(T^2)$. Let $C$ be a closed subset of $M_\bR$ which satisfies $\MS(p_!\bC_C)\subset \Ls$. Then $\Hom(F,p_!\bC_C)$ is perfect.
\end{lemma}
\begin{proof}
	For an open set $U \subset T^2$, the morphism
	\begin{equation}
	\Hom(F,p_!\bC_C)
	\lto 
	\Hom(F(U),(p_!\bC_C)(U))
	\end{equation}
	induces the dual morphism 
	\begin{equation}
	\Hom(F(U),(p_!\bC_C)(U))^\vee
	\lto 
	\Hom(F,p_!\bC_C)^\vee.
	\end{equation}
	Here by the boundedness of $C$, $(p_!\bC_C)(U)$ is finite-dimensional.
	Hence we obtain an isomorphism
	\begin{equation}
	\Hom(F(U),(p_!\bC_C)(U))^\vee
	\otimes (p_!\bC_C)(U)
	\simeq 
	\Hom(\Hom(F(U),(p_!\bC_C)(U)),(p_!\bC_C)(U)).
	\end{equation}
	Combining this with the evaluation map
	\begin{equation}
	F(U) \lto 
	\Hom(\Hom(F(U),(p_!\bC_C)(U)),(p_!\bC_C)(U)),
	\end{equation}
	we get a morphism
	\begin{equation}
	e_C(U) \colon F(U)
	\lto 
	\Hom(F,p_!\bC_C)^\vee
	\otimes 
	(p_!\bC_C)(U)
	\end{equation}
	and $e_C \in \Hom(F,\Hom(F,p_!\bC_C)^\vee \otimes p_!\bC_C)$.
	Here $\Hom(F,p_!\bC_C)^\vee \simeq \colim_V V$, where $V$ ranges through the family of finite-dimensional subspaces of $\Hom(F,p_!\bC_C)^\vee$.
	Since tensor products commute with colimits, one has an isomorphism
	\begin{equation}
	(\colim_V V) \otimes p_!\bC_C
	\simeq 
	\colim_V (V \otimes p_!\bC_C).
	\end{equation}
	Moreover by the compactness of $F$, we have
	\begin{equation}
	\begin{split}
	\Hom(F,\Hom(F,p_!\bC_C)^\vee \otimes p_!\bC_C)
	& \simeq 
	\Hom(F,\colim_V (V \otimes p_!\bC_C)) \\
	& \simeq 
	\colim_V \Hom(F,V \otimes p_!\bC_C).
	\end{split}
	\end{equation}
	Therefore there is a finite-dimensional subspace $V_0$ satisfying $e_S \in \Hom(F,V_0 \otimes p_!\bC_C)$.
	By construction, $\langle e_C,\varphi \rangle=\varphi$ for any $\varphi \in \Hom(F,p_!\bC_C)$, which implies that $\Hom(F,p_!\bC_C)^\vee$ is finite-dimensional.
\end{proof}

To prove Proposition~\ref{keyprop}, we give some explanation of constant sheaves on locally closed polyhedral sets.
\begin{definition}
	\begin{enumerate}
		\item A locally closed subset $Z$ of $M_\bR$ is a {\em locally closed rational polyhedral set} if  there exists $\rho_1,\dots, \rho_r\in N$ and $c_1,\dots,c_r\in \bZ$ for some $r$ such that 
		\begin{equation}
		Z=\bigcap_{i=1}^{i'} H_{\rho_i\geq c_i}\cap\bigcap_{i=i'+1}^rH_{-\rho_i> c_i} 
		\end{equation} 
		for some $i'$. 
		We sometimes omit ``rational" for short.
		\item  A locally closed polyhedral set $Z$ is {\em $\Ls$-polyhedral} if $\rho_i\in \Sigma(1)$ for any $i\leq i'$ and $-\rho_i\in \Sigma(1)$ for any $i>i'$ in the above expression. 
		\item A {\em face} of a locally closed polyhedral set is a face of the closure, which is a polyhedron.
		\item For a face $f$ of a locally closed polyhedral set $Z$, there exists the subset $R\subset \{1,\dots,r\}$ such that $f\subset \rho_i^\perp$ for $i\in R$. Then we refer to $\{\rho_i\}_{\in R}$ as \emph{defining rays} of $f$.
	\end{enumerate} 
\end{definition}

For $\dim \Sigma=2$, we define further notions. 
Let $Z\subset \bR^2$ be a locally closed polyhedral set and $v$ be a vertex of $Z$. 
Let $f_0$ and $f_1$ be facets of $Z$ containing $v$ and $\rho_0$ and $\rho_1$ be defining rays of $f_0$ and $f_1$, respectively.

We define the length of $v$ as
\begin{equation}
l(v):=\# \lc \rho\in \Sigma(1)\relmid \rho\subset \Int(\Cone(f_1,f_2))\rc.
\end{equation} 
Note that if $l(v)=0$ for any vertices of $Z$, then $Z$ is $\Ls$-polyhedral.

\begin{definition}
	For a $\Ls$-polyhedral set $Z$, a face $f$ of $Z$ is a \emph{$\Ls$-face} (or belongs to $\Ls$) if there exists a point $x$ in $\Int(f)$ where the microsupport at $x$ of $\bC_Z$ is in ${\Ls}$. Otherwise, we say $f$ is \emph{non-$\Ls$}.
\end{definition}

\begin{proof}[Proof of Proposition~\ref{keyprop}]
	We shall show that $\Gamma(\Int(C_x), \bD\iota_{C_x}^{-1}p^{-1}F)$ is perfect for any compact object $F\in \lSh_\Ls(T^2)$. 
	We will construct an object $\mC_x\in \cSh_\Ls(T^2)$ such that 
	\begin{equation}
	\hom(\mC_x, G)\simeq \hom(\bC_{\Int(C_x)},G)\simeq \Gamma(\Int(C_x), G),
	\end{equation} 
	where $G:= \bD\iota_{C_x}^{-1}p^{-1}F$. By Lemma~\ref{okcase}, this completes our proof. Note that $\MS(G)\subset \tLs$.
	
	If $\dim C_x=1$, after restricting to the affine span of $C_x$, we can take $\mC_x$ as the constant sheaf on the minimal open segment connecting points in $M$ which contains $C_x$.
	
	Hence, we assume $\dim C_x=2$. 
	Then non-$\Ls$ faces of $C_x$ are vertices. 
	Suppose that a vertex $v$ of $C_x$ is non-$\Ls$. 
	Let $\rho_0, \rho_1\in \Sigma(1)$ be the defining rays of $v$. 
	If $\Int(\Cone(\rho_0,\rho_1))$ contains no elements of $\Sigma(1)$ (i.e., $l(v)=0$), the smoothness assumption implies that $v$ belongs to $\Ls$. 
	This contradicts the assumption. 
	Hence, we can take $\rho_2 \in \Sigma(1)$ such that $\Int(\Cone(\rho_0,\rho_2))$ contains no $\Sigma(1)$'s.
	Note that a sufficiently small neighborhood of $v$ is a translation of a neighborhood of $0$ of $H_{\rho_0>0}\cap H_{\rho_1>0}$. We set
	\begin{equation}
	D_1:= H_{\rho_0\geq 0}\cap H_{\rho_1\geq 0}\cap H_{\rho_2<c},
	\end{equation} 
	where $c$ is the smallest integer which makes $D_1$ non-empty. Then $D_1$ is a 2-dimensional locally closed simplex such that those 3 vertices are 
	\begin{enumerate}
		\item $v$;
		\item a $\Ls$-vertex formed by $\rho_0$ and $\rho_2$, and
		\item a vertex $v'$ with $l(v')<l(v)$.
	\end{enumerate} 
	Thus, there exists an extension map $\bC_{D_1}\rightarrow \bC_{\Int(C_x)}$ (cf.\ \cite{STZ}). 
	We can repeat this process for other non-$\Ls$ vertices of $C_x$ and $D_1$.
	Moreover, the sum of lengths of vertices strictly decreases in this inductive process. 
	Hence, this process will stop in finitely many steps.
\end{proof}

Now we prove Conjecture~\ref{conj:gccc} (or equivalently Conjecture~\ref{conj:qccc}) for a smooth not necessarily complete 2-dimensional fan $\Sigma$.

\begin{theorem}
	If $\dim \Sigma=2$, Conjecture~\ref{conj:gccc} is true. 
	In other words, $\kappa_\Sigma$ induces a quasi-equivalence
	\begin{equation}
	\kappa_\Sigma \colon \Coh X_\Sigma\xrightarrow{\simeq} \wSh_{\Ls}(T^2)
	\end{equation}
	for any smooth (not necessarily complete) fan $\Sigma$.
\end{theorem}

\begin{proof}
	By completion and induction, it suffices to show the following:
	for a smooth complete fan $\tl \Sigma$ with $\dim \tl \Sigma=2$ and a 1-dimensional cone $\rho$ in $\tl \Sigma$, 
	one has the quasi-equivalence
	\begin{align}
	\kappa_{\tl{\Sigma}_c^\rho} \colon 
	\Coh(X_{\tl{\Sigma}_c^\rho}) \xrightarrow{\simeq} 
	\wSh_{\Lambda_{\tl{\Sigma}^\rho_c}}(T^2). \label{qe:2d1}
	\end{align}
	Here we used the same notation as in Section~\ref{sec:main}.
	By taking ind-categories in the equivalence of Theorem~\ref{thm:kuwagaki}, 
	by Theorem~\ref{compactness}, we see that Conjecture~\ref{conj:qccc} for $\Sigma$ holds.
	Hence the theorem follows from Theorem~\ref{main}.
\end{proof}

\section*{Acknowledgements}
The authors would like to thank Fumihiko~Sanda for teaching us the paper \cite{Seidelnat}. They also thank Zachary~Sylvan, Dmitry~Vaintrob, and Eric~Zaslow for discussions and comments, and Sheel~Ganatra for teaching us references for categorical localization of Fukaya categories. 
The first-named author is also grateful to Professor Pierre~Schapira for helpful comments.
The second-named author also thanks Professor Kazushi~Ueda for continuous encouragement. This work was supported by the Program for Leading Graduate Schools, MEXT, Japan. Both authors were supported
by the Grant-in-Aid for JSPS fellows.

\end{document}